\documentclass[12pt]{amsart}
 
\pdfoutput=1
\usepackage[margin=1in]{geometry}
\usepackage{amsmath,amsthm,amssymb,amsrefs,bbm,color,esint,esvect,float,graphicx,mathrsfs}

\usepackage{epstopdf}
\AppendGraphicsExtensions{.tif}

 \newtheorem{thm}{Theorem}[section]
  \newtheorem{remark}[thm]{Remark}
 \newtheorem{lemma}[thm]{Lemma}

\newtheorem{defin}[thm]{Definition}
\newtheorem{example}[thm]{Example}
\newtheorem{open}[thm]{Open Question}

\begin{document}

\title[Dimensional weak-type $(1,1)$ bound for Riesz transforms]{On the dimensional weak-type $(1,1)$ bound for Riesz transforms}
\author{Daniel Spector}
\address{Daniel Spector\hfill\break\indent 
Nonlinear Analysis Unit\hfill\break\indent 
Okinawa Institute of Science and Technology Graduate University\hfill\break\indent
1919-1 Tancha, Onna-son, Kunigami-gun\hfill\break\indent 
Okinawa, Japan}
\email{daniel.spector@oist.jp}
\author{Cody B. Stockdale}
\address{Cody B. Stockdale\hfill\break\indent 
 Department of Mathematics and Statistics\hfill\break\indent 
 Washington University in St. Louis\hfill\break\indent 
 One Brookings Drive\hfill\break\indent 
 St. Louis, MO, 63130 USA}
\email{codystockdale@wustl.edu}

\begin{abstract}
Let $R_j$ denote the $j^{\text{th}}$ Riesz transform on $\mathbb{R}^n$. We prove that there exists an absolute constant $C>0$ such that
\begin{align*}
	|\{|R_jf|>\lambda\}|\leq C\left(\frac{1}{\lambda}\|f\|_{L^1(\mathbb{R}^n)}+\sup_{\nu} |\{|R_j\nu|>\lambda\}|\right)
\end{align*}
for any $\lambda>0$ and $f \in L^1(\mathbb{R}^n)$, where the above supremum is taken over measures of the form $\nu=\sum_{k=1}^Na_k\delta_{c_k}$ for $N \in \mathbb{N}$, $c_k \in \mathbb{R}^n$, and $a_k \in \mathbb{R}^+$ with $\sum_{k=1}^N a_k \leq 16\|f\|_{L^1(\mathbb{R}^n)}$.  This shows that to establish dimensional estimates for the weak-type $(1,1)$ inequality for the Riesz tranforms it suffices to study the corresponding weak-type inequality for Riesz transforms applied to a finite linear combination of Dirac masses. We use this fact to give a new proof of the best known dimensional upper bound, while our reduction result also applies to a more general class of Calder\'on-Zygmund operators.
\end{abstract}

\maketitle

\section{Introduction}

Let $\mathbb{R}^n$ denote the Euclidean space of $n$ dimensions and, for $j \in \{1,2,\ldots,n\}$, define the $j^{\text{th}}$ Riesz transform of a function $f \in C^\infty_c(\mathbb{R}^n)$ by 
\begin{align*}
	R_jf(x):=\frac{\Gamma\left(\frac{n+1}{2}\right)}{\pi^{\frac{n+1}{2}}} \;p.v. \int_{}\frac{x_j-y_j}{|x-y|^{n+1}}f(y)\,dy.
\end{align*}

It is a classical result of Calder\'on and Zygmund \cite{CZ} that the Riesz transforms extend as bounded operators on $L^p(\mathbb{R}^n)$ for $1<p<+\infty$ and from $L^1(\mathbb{R}^n)$ into weak $L^1(\mathbb{R}^n)$.  More precisely, they show that for $1<p<+\infty$, one has the strong-type $(p,p)$ inequality
\begin{align}
\|R_j f\|_{L^p(\mathbb{R}^n)} \leq C_p(n) \|f\|_{L^p(\mathbb{R}^n)} \label{lp}
\end{align}
for all $f \in L^p(\mathbb{R}^n)$, while if $p=1$, their results imply the weak-type $(1,1)$ inequality
\begin{align}
\sup_{\lambda>0} \lambda |\{ |R_j f|>\lambda \}| \leq C_1(n) \|f\|_{L^1(\mathbb{R}^n)}\label{wl1}
\end{align}
for all $f \in L^1(\mathbb{R}^n)$.

The method in \cite{CZ} is to first establish a slight variant of the inequality \eqref{wl1} and then to prove \eqref{lp} by an interpolation argument.  The proof of \eqref{wl1}, in turn, is argued by the (subsequently termed) Calder\'on-Zygmund decomposition, from which one obtains an exponential dependence in the dimension of the constant $C_1(n)$.  Naturally, in this argument $C_p(n)$ inherits this dependence.  However, the constants $C_p(n)$ can actually be taken to be dimension free, as was first shown by E. M. Stein in \cite{Stein83} (and can even be explicitly computed, see \cite{IwaniecMartin}).  It was a question of Stein  \cite[Problem b on p.~203]{Stein86} whether the constant $C_1(n)$ can also be taken to be dimension free.   At present the best result in this direction is that of P. Janakiraman, who showed in \cite{J2004} that \eqref{wl1} holds with $C_1(n)=c \log(n)$ for some absolute constant $c>0$.

These questions parallel a similar line of research concerning dimensional estimates for maximal functions, including the centered Hardy-Littlewood maximal function:
\begin{align*}
\mathcal{M}f(x):= \sup_{r>0} \fint_{B(x,r)} |f(y)|\;dy.
\end{align*}
In particular, it was asserted by Stein in \cite{Stein82} that if $1< p \leq +\infty$, then
\begin{align*}
\|\mathcal{M}f\|_{L^p(\mathbb{R}^n)} \leq C'_p \|f\|_{L^p(\mathbb{R}^n)}
\end{align*}
for all $f \in L^p(\mathbb{R}^n)$, with a constant $C'_p>0$ independent of $n$.  The proof of this fact appeared in a subsequent paper in collaboration with J. O. Str\"omberg, \cite{SteinStromberg}. Here, they also proved the dimensional weak-type $(1,1)$ estimate
\begin{align}
\sup_{\lambda>0} \lambda |\{ \mathcal{M}f>\lambda \}| \leq C'_1(n)\|f\|_{L^1(\mathbb{R}^n)} \label{mwl1}
\end{align}
for all $f \in L^1(\mathbb{R}^n)$, where $C'_1(n)=cn$ for some absolute constant $c>0$. 

At present, it remains unknown whether the linear dependence in \eqref{mwl1} is optimal.  One possible approach to an improvement to the result of Stein and Str\"omberg would be to establish a dimensional bound in the inequality
\begin{align}\label{hlpointmasses}
\sup_{\lambda>0} \lambda |\{ \mathcal{M}\nu>\lambda \}| \leq C'_1(n)\|\nu\|_{M_b(\mathbb{R}^n)}
\end{align}
over all bounded measures $\nu$ of the form $\nu = \sum_{k=1}^N \delta_{c_k}$ for any $N \in \mathbb{N}$ and $c_k \in \mathbb{R}^n$, where $\mathcal{M}\nu(x) := \sup_{r>0}\frac{|\nu|(B(x,r))}{|B(x,r)|}$ and $\|\nu\|_{M_b(\mathbb{R}^n)}$ denotes the total variation of $\nu$.  Indeed, by a result of M. de Guzm\'an \cite[Theorem 4.1.1]{deguzman} the two constants are comparable (and can even be taken to be the same, see \cite{TMSoria}).  This perspective has proven useful in obtaining lower bounds, that is, in ruling out the possibility of a dimension free constant, for the centered maximal function associated to cubes in $\mathbb{R}^n$.  In particular, in \cite{A2011} J. M. Aldaz  establishes that the weak-type $(1,1)$ bound for this operator tends to infinity with the dimension by considering the operator applied to Dirac masses (see also A. S. Iakovlev and Str\"omberg \cite{IS2013}, who subsequently improved Aldaz's result with the explicit estimate $C_1'(n) \geq cn^{\frac{1}{4}}$).   In general, as the study of such estimates on sums of Dirac masses presents the possibility for more explicit computations, the inequality \eqref{hlpointmasses} seems to be a simpler formulation of the problem of understanding dimensional bounds.


The main result of this paper is an analogue of de Guzm\'an's result for the Riesz transforms, the following
\begin{thm}\label{pointmassreduction}
There exists an absolute constant $C>0$ such that
\begin{align*}
	|\{|R_jf|>\lambda\}|\leq C\left( \frac{1}{\lambda}\|f\|_{L^1(\mathbb{R}^n)}+\sup_{\nu} |\{|R_j\nu|>\lambda\}|\right)
\end{align*}
for any $\lambda>0$ and $f \in L^1(\mathbb{R}^n)$, where the above supremum is taken over measures of the form $\nu=\sum_{k=1}^Na_k\delta_{c_k}$ for $N \in \mathbb{N}$, $c_k \in \mathbb{R}^n$, and $a_k \in \mathbb{R}^+$ with $\sum_{k=1}^N a_k \leq 16\|f\|_{L^1(\mathbb{R}^n)}$.
\end{thm}

\noindent
Above, $R_j\nu(x):= \frac{\Gamma\left(\frac{n+1}{2}\right)}{\pi^{\frac{n+1}{2}}} \;p.v. \int_{}\frac{x_j-y_j}{|x-y|^{n+1}}\,d\nu(y)$. Theorem \ref{pointmassreduction} says that to establish a dimensional weak-type $(1,1)$ estimate for $R_j$, it suffices to prove such an estimate for the operator applied to a finite linear combination of Dirac masses.

\begin{remark}
Theorem \ref{pointmassreduction} actually holds for a more general class of singular integral operators including the second order Riesz transforms; see the precise assumptions in Section \ref{preliminaries} and more general result given in Theorem \ref{pointmassreduction1} below.
\end{remark}

Our proof of Theorem \ref{pointmassreduction} is based on the approach of F. Nazarov, S. Treil, and A. Volberg in \cite{NTV1998}, and builds upon the further work of the second named author in \cites{GS2019, S2020, S2019, SW2019}.  While a direct application of these arguments yields exponential growth in the dimension, we here make suitable modifications and a careful accounting to remove this dependence.  One can also recover the dimensional dependence proved by Janakiraman in \cite{J2004} by our



\begin{thm}\label{pointmassweaktype}
There exists an absolute constant $C>0$ such that 
\begin{align*}
	\sup_{\lambda>0}\lambda|\{|R_j\nu|>\lambda\}| \leq C\log(n) \|\nu\|_{M_b(\mathbb{R}^n)}
\end{align*}
for all measures $\nu$ of the form $\nu = \sum_{k=1}^N a_k\delta_{c_k}$ with $c_k \in \mathbb{R}^n$ and $a_k \in \mathbb{R}^+$.
\end{thm}

In light of Stein's dimensionless weak-type $(1,1)$ question for $R_j$ from \cite{Stein86}, this naturally leads one to pose
\begin{open}\label{dimensionlesspointmass}
Does there exist an absolute constant $C>0$ such that 
\begin{align*}
	\sup_{\lambda>0}\lambda|\{|R_j\nu|>\lambda\}| \leq C \|\nu\|_{M_b(\mathbb{R}^n)}
\end{align*}
for all $\nu \in M_b(\mathbb{R}^n)$ of the form $\nu = \sum_{k=1}^N a_k\delta_{c_k}$?
\end{open}
\noindent
In particular, a solution to Open Question \ref{dimensionlesspointmass} together with Theorem \ref{pointmassreduction} would imply an affirmative answer to Stein's question. 

This reduction to the study of Riesz transforms applied to Dirac masses - for which one has explicit formulas in terms of rational functions -  leads to some interesting phenomena.  For example, one finds that in the case $\nu=a\delta_c$,
\begin{align*}
R_j\nu(x) = \frac{\Gamma\left(\frac{n+1}{2}\right)}{\pi^{\frac{n+1}{2}}} a\frac{x_j-c_j}{|x-c|^{n+1}},
\end{align*}
and therefore
 \begin{align*}
 |\{|R_j\nu|>\lambda\}| &\leq \bigg|\bigg\{\frac{\Gamma\left(\frac{n+1}{2}\right)}{\pi^{\frac{n+1}{2}}} \frac{|a|}{|x-c|^n}>\lambda\bigg\}\bigg|
 \\&=|B(0,1)|\frac{\Gamma\left(\frac{n+1}{2}\right)}{\pi^{\frac{n+1}{2}}} \frac{1}{\lambda}\|\nu\|_{M_b(\mathbb{R}^n)}
\end{align*}
for any $\lambda>0$.  A simple computation (see \cite{J2004}, p.~553) then shows that
\begin{align*}
|B(0,1)|\frac{\Gamma\left(\frac{n+1}{2}\right)}{\pi^{\frac{n+1}{2}}} =  \frac{2\pi^{n/2}}{n\Gamma\left(\frac{n}{2}\right)}\frac{\Gamma\left(\frac{n+1}{2}\right)}{\pi^{\frac{n+1}{2}}} \approx \frac{1}{\sqrt{n}}
\end{align*} 
for $n$ large, and so the bound tends to zero as $n$ tends to infinity in the case of one Dirac mass!  Note that this is in contrast to the case of the Hardy-Littlewood maximal function, where one has constant dependence on the dimension for a single Dirac mass.

 Of course, we must understand what happens when there are multiple Dirac masses, though the geometry quickly becomes quite complicated.  The question in one dimension may yield some insight into the effects of cancellation.  In particular, in the case $n=1$ and $\nu=a_1\delta_{c_1}+a_2\delta_{c_2}$ for $a_1,a_2 >0$ (we can always take $a_k>0$ by separating the positive and negative terms and doubling the constant), one can explicitly compute the level sets of $H\nu$ (as $R_1=H$, the Hilbert transform) and show
\begin{align}\label{additive}
|\{|H\nu|>\lambda\}| =\frac{2}{\pi} \frac{1}{\lambda}\|\nu\|_{M_b(\mathbb{R}^n)} =  \frac{|B(0,1)|}{\pi} \frac{1}{\lambda}\|\nu\|_{M_b(\mathbb{R}^n)}
\end{align}
for any $\lambda>0$.  This is a simple calculation, though with only a slightly more subtle argument, such an equality -- independent of the number of Dirac masses -- had already been proved in 1946!  Precisely, in \cite{Loomis1946}  L. Loomis established the inequality \eqref{additive} for all $\nu \in M_b(\mathbb{R})$ of the form $\nu = \sum_{k=1}^N a_k\delta_{c_k}$ with  $a_k>0$.   It seems that a careful consideration of the geometry of Euclidean space may yield some insight into this question in higher dimensions, and from this of course, an answer to the question of Stein.

The plan of the paper is as follows. In Section \ref{preliminaries}, we introduce the class of operators we work with and discuss the main examples of Riesz transforms and second order Riesz transforms. In Section \ref{lemmas}, we collect some relevant lemmas. Finally, in Section \ref{mainresults}, we prove the main results.  We begin with a result more general than Theorem \ref{pointmassreduction}, our Theorem \ref{pointmassreduction1}, from which Theorem \ref{pointmassreduction} follows immediately.  We then conclude with a proof of Theorem \ref{pointmassweaktype}.


\section{Preliminaries}\label{preliminaries}

\begin{defin}\label{SIO}
Assume that $K:\mathbb{R}^n \setminus \{0\}\rightarrow \mathbb{C}$ satisfies $K(x)=\frac{\Omega(x)}{|x|^n}$, where $\Omega$ is a function such that
\begin{enumerate}
	\item 
		$$
			\Omega(x)=\Omega\bigg(\frac{x}{|x|}\bigg)=\Omega(\delta x),
		$$
		 for $x \neq 0$ and $\delta >0$,
	\item 
		$$
			\int_{S^{n-1}}\Omega(\theta)\,d\sigma(\theta)=0,
		$$ 
		where $\sigma$ denotes surface measure on $S^{n-1}$, and 
	\item there exists an absolute constant $C>0$ such that  
		$$
			\int_{S^{n-1}}|\Omega(\theta-\xi\delta)-\Omega(\theta)|\,d\sigma(\theta)\leq C n\delta\int_{S^{n-1}}|\Omega(\theta)|\,d\sigma(\theta)
		$$
	for $\xi \in S^{n-1}$ and $0<\delta<\frac{1}{n}$
\end{enumerate}
Define $T$ to be the singular integral operator associated to a kernel $K$ as described above:
$$
	Tf(x):=p.v.\int_{}K(x-y)f(y)\, dy \equiv \lim_{\varepsilon \to 0} \int_{|x-y| \geq \varepsilon}K(x-y)f(y)\, dy ,
$$
for $f \in C_c^{\infty}(\mathbb{R}^n)$. 
\end{defin}

\begin{example}\label{Riesz1}
The Riesz transforms $R_j$ are examples of such singular integral operators with 
\begin{align*}
\Omega(x)=\frac{\Gamma\left(\frac{n+1}{2}\right)}{\pi^{\frac{n+1}{2}}}\frac{x_j}{|x|}.
\end{align*}   
One can show that 
\begin{align*}
\int_{S^{n-1}}|\Omega(\theta)|\, d\sigma(\theta) = \frac{2}{\pi}
\end{align*}
and
\begin{align*}
			\int_{S^{n-1}}|\Omega(\theta-\xi\delta)-\Omega(\theta)|\,d\sigma(\theta)\leq C \sqrt{n}\delta\int_{S^{n-1}}|\Omega(\theta)|\,d\sigma(\theta)
\end{align*}
	for $\xi \in S^{n-1}$ and $0<\delta<\frac{1}{n}$ (see, e.g. \cite[p.~554]{J2004}).  We observe here that a slight improvement can be made in Janakiraman's computation.  In particular, one has
\begin{align*}
&\left| \frac{\partial}{\partial x_j} \frac{x_j}{|x|}\right| = \left| \frac{1}{|x|}-\frac{x_j^2}{|x|^3}\right|\quad\text{and} \\
&\left| \frac{\partial}{\partial x_i} \frac{x_j}{|x|}\right| = \left| \frac{x_ix_j}{|x|^3}\right|  \quad i \neq j,
\end{align*}
which implies
\begin{align*}
\left|\nabla \frac{x_j}{|x|}\right|^2 &= \sum_{i\neq j} \left| \frac{x_ix_j}{|x|^3}\right|^2 + \left| \frac{1}{|x|}-\frac{x_j^2}{|x|^3}\right|^2\\
&=\frac{x_j^2}{|x|^4}-\frac{x_j^4}{|x|^6}+\frac{1}{|x|^2}-\frac{2x_j^2}{|x|^4}+\frac{x_j^4}{|x|^4}\\
&=\frac{1}{|x|^2}-\frac{x_j^2}{|x|^4},
\end{align*}
and therefore
\begin{align*}
\left|\nabla \frac{x_j}{|x|}\right| \leq \frac{1}{|x|},
\end{align*}
without the need for $\sqrt{n}$ in the numerator as in \cite{J2004}.
\end{example}

\begin{example}\label{Riesz2}
The higher order Riesz transforms, $R_{ij}$, are also examples included in the above framework.  In particular, we compute
\begin{align*}
R_{ij}f(x):= \frac{1}{\gamma(2)} \lim_{\varepsilon \to 0} \int_{|x-y|\ge \varepsilon} \frac{\partial^2 }{\partial x_i x_j}|x-y|^{-n+2} f(y)\;dy
\end{align*}
where $\gamma(2)=\pi^{n/2}2^2/\Gamma(n/2-1)$.   The observation that 
\begin{align*}
\frac{\partial }{\partial x_j}|x-y|^{-n+2}&=(-n+2) |x-y|^{-n+1} \frac{x_j-y_j}{|x-y|}
\end{align*}
implies that
\begin{align*}
\frac{\partial^2 }{\partial x_i x_j}|x-y|^{-n+2} &= (-n+2)(-n) |x-y|^{-n-1} (x_j-y_j)\frac{x_i-y_i}{|x-y|}.
\end{align*}
for $i\neq j$. Meanwhile, in the case $i=j$, one has
\begin{align*}
\frac{\partial^2 }{\partial x^2_j}|x-y|^{-n+2} &= (-n+2)(-n) |x-y|^{-n-1} (x_j-y_j)\frac{x_j-y_j}{|x-y|}+(-n+2)|x-y|^{-n}.
\end{align*}
Therefore one obtains
\begin{align*}
R_{ij}f(x) = p.v.\; \int \frac{\Omega_{ij}(x-y)}{|x-y|^n}f(y)\;dy
\end{align*}
for
\begin{align*}
\Omega_{ij}(x) &= \frac{\Gamma(n/2+1)}{\pi^{n/2}}\frac{x_ix_j}{|x|^2} \quad i\neq j, \\
\Omega_{jj}(x) &= \frac{\Gamma(n/2+1)}{\pi^{n/2}} \left(\frac{x_j^2}{|x|^2} -\frac{1}{n}\right),
\end{align*}
where we have used that
\begin{align*}
\frac{(n-2)n}{\gamma(2)} &= n/2 \times(n/2-1) \times \Gamma(n/2-1)/\pi^{n/2} \\
&=\frac{\Gamma(n/2+1)}{\pi^{n/2}}.
\end{align*}
We next observe that for $i \neq j$
\begin{align*}
\int_{S^{n-1}} |\Omega_{ij}(\theta)| \,d\sigma(\theta) &\leq \frac{\Gamma(n/2+1)}{\pi^{n/2}} \frac{1}{2} \int_{S^{n-1}} \theta_i^2+\theta_j^2 \, d\sigma(\theta) \\
&=\frac{\Gamma(n/2+1)}{\pi^{n/2}} \frac{|S^{n-1}|}{n} \\
&=\frac{\Gamma(n/2+1)}{\pi^{n/2}} \frac{2\pi^{n/2}}{n\Gamma(n/2)} \\
&=1,
\end{align*}
while in the case $i=j$
\begin{align*}
\int_{S^{n-1}} |\Omega_{jj}(\theta)|\, d\sigma(\theta) &\leq \frac{\Gamma(n/2+1)}{\pi^{n/2}}  \int_{S^{n-1}} \left(\theta_j^2+1/n\right) \,  d\sigma(\theta) \\
&=2\frac{\Gamma(n/2+1)}{\pi^{n/2}} \frac{|S^{n-1}|}{n} \\
&=2\frac{\Gamma(n/2+1)}{\pi^{n/2}} \frac{2\pi^{n/2}}{n\Gamma(n/2)} \\
&=2.
\end{align*}
Finally, it remains to show
\begin{align*}
			\int_{S^{n-1}}|\Omega_{ij}(\theta-\xi\delta)-\Omega_{ij}(\theta)|\,d\sigma(\theta)\leq C \sqrt{n}\delta\int_{S^{n-1}}|\Omega_{ij}(\theta)|\,d\sigma(\theta)
\end{align*}
for $\xi \in S^{n-1}$ and $0<\delta<\frac{1}{n}$, for which it suffices to prove
\begin{align*}
\left| \nabla \frac{x_ix_j}{|x|^2}\right| &\leq \frac{c}{|x|}\quad\text{and}\\
\left| \nabla \frac{x_i^2}{|x|^2} \right| &\leq \frac{c'}{|x|}
\end{align*}
for some $c,c'>0$ independent of $n$, as Janakiraman's computation \cite[p.~553-554]{J2004} implies the desired result.

We first treat the case $i\neq j$.  To this end, observe that
\begin{align*}
\left| \frac{\partial}{\partial x_i} \frac{x_ix_j}{|x|^2} \right| = \left| \frac{x_j}{|x|^2} - \frac{2 x_i^2x_j}{|x|^4}\right|,
\end{align*}
while
\begin{align*}
\left| \frac{\partial}{\partial x_k} \frac{x_ix_j}{|x|^2} \right| = \left| \frac{2 x_ix_jx_k}{|x|^4}\right|.
\end{align*}
In particular
\begin{align*}
\left| \nabla \frac{x_ix_j}{|x|^2}\right|^2 = \sum_{k\neq i,j} \left| \frac{2 x_ix_jx_k}{|x|^4}\right|^2 +\left| \frac{x_j}{|x|^2} - \frac{2 x_i^2x_j}{|x|^4}\right|^2 + \left| \frac{x_i}{|x|^2} - \frac{2 x_j^2x_i}{|x|^4}\right|^2.
\end{align*}
However, 
\begin{align*}
\sum_{k\neq i,j} \left| \frac{2 x_ix_jx_k}{|x|^4}\right|^2 = \frac{4x_i^2x_j^2}{|x|^6} -\frac{4 x^4_ix_j^2}{|x|^8}-\frac{4 x_i^2x_j^4}{|x|^8}
\end{align*}
and
\begin{align*}
\left| \frac{x_j}{|x|^2} - \frac{2 x_i^2x_j}{|x|^4}\right|^2 &= \frac{x_j^2}{|x|^4} \left(1-\frac{4x_i^2}{|x|^2}+ \frac{4x_i^4}{|x|^4}\right), \\
 \left| \frac{x_i}{|x|^2} - \frac{2 x_j^2x_i}{|x|^4}\right|^2 &=\frac{x_i^2}{|x|^4} \left(1-\frac{4x_j^2}{|x|^2}+ \frac{4x_j^4}{|x|^4}\right), \\
\end{align*}
so that
\begin{align*}
\left| \nabla \frac{x_ix_j}{|x|^2}\right|^2 &= \frac{4x_i^2x_j^2}{|x|^6} -\frac{4 x^4_ix_j^2}{|x|^8}-\frac{4 x_i^2x_j^4}{|x|^8} +\frac{x_j^2}{|x|^4} \left(1-\frac{4x_i^2}{|x|^2}+ \frac{4x_i^4}{|x|^4}\right)+\frac{x_i^2}{|x|^4} \left(1-\frac{4x_j^2}{|x|^2}+ \frac{4x_j^4}{|x|^4}\right) \\
&= \frac{x_j^2}{|x|^4} +\frac{x_i^2}{|x|^4} - \frac{4x_i^2x_j^2}{|x|^6}.
\end{align*}
This shows one can take $c= \sqrt{5}$ (though a more clever observation here could possibly do better).

Finally for the case $i=j$, we have
\begin{align*}
\left| \frac{\partial}{\partial x_i} \frac{x_i^2}{|x|^2} \right| = \left| \frac{2x_i}{|x|^2} - \frac{2x_i^3}{|x|^4}\right|,
\end{align*}
while
\begin{align*}
\left| \frac{\partial}{\partial x_k} \frac{x_i^2}{|x|^2} \right| = \left| \frac{2 x_i^2x_k}{|x|^4}\right|.
\end{align*}
Therefore
\begin{align*}
\left| \nabla  \frac{x_i^2}{|x|^2}  \right|^2 = \sum_{k\neq i} \left| \frac{2 x_i^2x_k}{|x|^4}\right|^2 +\left| \frac{2x_i}{|x|^2} - \frac{2x_i^3}{|x|^4}\right|^2.
\end{align*}
However in a similar way one computes
\begin{align*}
 \sum_{k\neq i} \left| \frac{2 x_i^2x_k}{|x|^4}\right|^2 = \frac{4x_i^4}{|x|^6} -\frac{4 x_i^6}{|x|^8}
\end{align*}
and
\begin{align*}
\left| \frac{2x_i}{|x|^2} - \frac{2x_i^3}{|x|^4}\right|^2 &= \frac{4x_i^2}{|x|^4}\left(1-\frac{x_i^2}{|x|^2}+\frac{x_i^4}{|x|^4}\right).
\end{align*}
Thus
\begin{align*}
\left| \nabla \frac{x_i^2}{|x|^2} \right|^2  &= \frac{4x_i^4}{|x|^6} -\frac{4 x_i^6}{|x|^8} +\frac{4x_i^2}{|x|^4}\left(1-\frac{x_i^2}{|x|^2}+\frac{x_i^4}{|x|^4}\right)\\
&=\frac{4x_i^2}{|x|^4},
\end{align*}
so that $c'=2$ is sufficient.
\end{example}


\section{Lemmas}\label{lemmas}
The following Lemma is a dimensional modification of the usual Whitney decomposition.  We here adapt the argument given in \cite[p.~609]{Grafakos1}.
\begin{lemma}\label{Whitney}
If $U \subseteq \mathbb{R}^n$ is an open set, then we can write $U = \bigcup_{k=1}^{\infty} Q_k$, a disjoint union of dyadic cubes satisfying 
$$
	(2n-1)\text{diam}(Q_k)\leq \text{dist}(Q_k, \mathbb{R}^n\setminus U).
$$
\end{lemma}

\begin{proof}
Set 
$$
	U_k:= \{x \in U : 2n\sqrt{n}2^{-k}\leq \text{dist}(x, \mathbb{R}^n\setminus U) < 4n\sqrt{n}2^{-k}\}.
$$ 
Denote the dyadic cubes with side length $2^{-k}$ by $\mathcal{D}_k$ and define 
$$
	\mathcal{F}_k':=\{Q \in \mathcal{D}_k : Q \cap U_k \neq \emptyset\},\quad \mathcal{F}':= \bigcup_{k \in \mathbb{Z}} \mathcal{F}_k',\quad \text{and}
$$
$$
	\mathcal{F}:=\{Q \in \mathcal{F}': Q \text{ is maximal with respect to inclusion}\}.
$$ 
Then $\mathcal{F}$ is a countable collection of pairwise disjoint dyadic cubes and $U=\bigcup_{Q \in \mathcal{F}}Q$. Moreover, for $Q \in \mathcal{F}$, pick a point $x \in U_k \cap Q$ for some $k \in \mathbb{Z}$. Then 
\begin{align*}
n\text{diam}(Q)&=n\sqrt{n}2^{-k}\\
&\leq \text{dist}(x,\mathbb{R}^n\setminus U)-n\sqrt{n}2^{-k}\\
&=\text{dist}(x,\mathbb{R}^n\setminus U)-n\sqrt{n}\ell(Q)\\
&=\text{dist}(x,\mathbb{R}^n\setminus U)-(\sqrt{n}\ell(Q)+(n-1)\sqrt{n}\ell(Q))\\
&\leq \text{dist}(Q,\mathbb{R}^n\setminus U)-(n-1)\sqrt{n}\ell(Q)\\
&=\text{dist}(Q,\mathbb{R}^n\setminus U)-(n-1)\text{diam}(Q).
\end{align*}
Hence
$$
	(2n-1)\text{diam}(Q)\leq \text{dist}(Q, \mathbb{R}^n\setminus U).
$$
\end{proof}

\begin{lemma}\label{Hormander}
There exists an absolute constant $C_1>0$ such that for all $n\ge 2$, 
$$
	\int_{|x|>n|y|}|K(x-y)-K(x)|\,dx \leq C_1\|\Omega\|_{L^1(S^{n-1},\sigma)}.
$$
\end{lemma}
\noindent Lemma \ref{Hormander} is precisely the claim in \cite[page 542]{J2004} and subsequently proved therein on pages 550--552.

\begin{lemma}\label{cancellation}
If $\mu$ is a signed Borel measure supported on $B(x,r)$ and $\mu(B(x,r))=0$ for some $x \in \mathbb{R}^n$ and $r>0$, then 
$$
	\int_{|x-y|>nr} |T\mu(y)|\,dy \leq C_1\|\Omega\|_{L^1(S^{n-1},\sigma)} \|\mu\|_{M_b(\mathbb{R}^n)}.
$$
\end{lemma}

\begin{proof}
Without loss of generality, suppose $x=0$. Since $\text{supp}\, \mu \subseteq B(0,r)$ and $\mu(B(0,r))=0$, 
$$
    |T\mu(y)| = \bigg|\int_{|z|<r}K(y-z)\,d\mu(z)\bigg|=\bigg|\int_{|z|<r} (K(y-z)-K(y))\,d\mu(z)\bigg|.
$$ 
Therefore, using Fubini's theorem and Lemma \ref{Hormander}, we see
\begin{align*}
\int_{|y|>nr}|T\mu(y)|\,dy &\leq \int_{|y|>nr}\int_{|z|<r}|K(y-z)-K(y)|\,d|\mu|(z)dy\\
&\leq\int_{|z|<r}\int_{|y|>n|z|}|K(y-z)-K(y)|\,dyd|\mu|(z)\\
&\leq C_1\|\Omega\|_{L^1(S^{n-1},\sigma)}\|\mu\|_{M_b(\mathbb{R}^n)}.
\end{align*}
\end{proof}

To recover Janakiranman's dimensional dependence result for the Riesz transforms, we will also need to consider the maximal truncation operator, $T^{\#}$, given by
\begin{align*}
T^{\#}f(x):=\sup_{\varepsilon>0} \left|\int_{|x-y|>\varepsilon}K(x-y)f(y)\,dy\right|
\end{align*}
for $f \in C_c^{\infty}(\mathbb{R}^n)$. The following lemma can be justified using the method of rotations, see \cite[Remark 5.2.9 on p.~ 341]{Grafakos1} for details.
\begin{lemma}\label{dimensionlessl2}
Let $T$ be a singular integral operator satisfying the conditions of Definition \ref{SIO} and further suppose that $\Omega$ is odd.  There exist absolute constants $C_2,C_3>0$ such that 
\begin{align*}
\|Tf\|_{L^2(\mathbb{R}^n)}\leq C_2\|\Omega\|_{L^1(S^{n-1},\sigma)}\|f\|_{L^2(\mathbb{R}^n)}
\end{align*}
and
\begin{align*}
\|T^{\#}f\|_{L^2(\mathbb{R}^n)}\leq C_3\|\Omega\|_{L^1(S^{n-1},\sigma)}\|f\|_{L^2(\mathbb{R}^n)}
\end{align*}
for all $f \in L^2(\mathbb{R}^n)$.
\end{lemma}
\noindent Note that Lemma \ref{dimensionlessl2} applies to the Riesz transforms since $\Omega(x)=\frac{\Gamma\left(\frac{n+1}{2}\right)}{\pi^{\frac{n+1}{2}}}\frac{x_j}{|x|}$ is odd.

\section{Main Results}\label{mainresults}


\begin{thm}\label{pointmassreduction1}
There exist absolute constants $C_4,C_5>0$ such that
\begin{align*}
	|\{|Tf|>\lambda\}|\leq \left(C_4+C_5\|\Omega\|_{L^1(S^{n-1},\sigma)}\right)\frac{1}{\lambda}\|f\|_{L^1(\mathbb{R}^n)}+2\sup_{\nu}|\{|T\nu|>\lambda\}|,
\end{align*}
for any $\lambda>0$ and $f \in L^1(\mathbb{R}^n)$, where the above supremum is taken over measures of the form $\nu=\sum_{k=1}^Na_k\delta_{c_k}$ for $N \in \mathbb{N}$, $c_k \in \mathbb{R}^n$, and $a_k \in \mathbb{R}^+$ with $\sum_{k=1}^N a_k \leq 16\|f\|_{L^1(\mathbb{R}^n)}$.
\end{thm}

\begin{proof}
Let $\lambda>0$ and $f\in L^1(\mathbb{R}^n)$ be given. By density, we may assume $f$ is a continuous function with compact support. First suppose that $f$ is nonnegative. Set 
$$
	U:=\left\{f>\lambda\right\}
$$ 
and apply Lemma \ref{Whitney} to write 
$$
	U=\bigcup_{k=1}^{\infty}Q_k,
$$ 
a disjoint union of dyadic cubes where 
$$
	(2n-1)\text{diam}(Q_k) \leq \text{dist}(Q_k,\mathbb{R}^n\setminus U).
$$ 
Put 
$$
	g:=f\chi_{\mathbb{R}^n\setminus U}, \quad\quad b:=f\chi_{U}, \quad\quad \text{and} \quad\quad b_k:=f\chi_{Q_k}.
$$ 
Then 
$$
	f=g+b=g+\sum_{k=1}^{\infty}b_k,
$$ 
where
\begin{enumerate}
\addtolength{\itemsep}{0.2cm}
\item[(1)] $\, \|g\|_{L^{\infty}(\mathbb{R}^n)}\leq \lambda$ and  $\|g\|_{L^1(\mathbb{R}^n)}\leq \|f\|_{L^1(\mathbb{R}^n)}$,
\item[(2)] \, the $b_k$ are supported on pairwise disjoint cubes $Q_k$ satisfying 
$$
	\sum_{k=1}^{\infty}|Q_k|\leq \frac{1}{\lambda}\|f\|_{L^1(\mathbb{R}^n)}, 
$$
and
\item[(3)] $\|b\|_{L^1(\mathbb{R}^n)}\leq \|f\|_{L^1(\mathbb{R}^n)}$.
\end{enumerate}
We begin the estimate with a standard quasi-subadditivity inequality we will repeat often in what follows.  In particular, the inclusion
\begin{align*}
\{|Tf|>\lambda\} \subseteq \left\{|Tg|>\frac{\lambda}{2}\right\} \cup \left\{|Tb|>\frac{\lambda}{2}\right\}
\end{align*} 
implies
\begin{align*}
	|\{|Tf|>\lambda\}|\leq \left|\left\{|Tg|>\frac{\lambda}{2}\right\}\right|+\left|\left\{|Tb|>\frac{\lambda}{2}\right\}\right|.
\end{align*}

To control the first term, we have by Chebyshev's inequality, Lemma \ref{dimensionlessl2}, and property (1) the estimate
\begin{align*}
\bigg|\bigg\{|Tg|>\frac{\lambda}{2}\bigg\}\bigg| &\leq \frac{4}{\lambda^2}\|Tg\|_{L^2(\mathbb{R}^n)}^2\\
&\leq \frac{4C_2^2}{\lambda^2}\|g\|_{L^2(\mathbb{R}^n)}^2\\
&\leq \frac{4C_2^2}{\lambda}\|g\|_{L^1(\mathbb{R}^n)}\\
&\leq \frac{4C_2^2}{\lambda}\|f\|_{L^1(\mathbb{R}^n)}.
\end{align*}

We now control the second term. For positive integers $N$, let $b^{(N)}$ denote the partial sum $\sum_{k=1}^N b_k$. We claim it suffices to obtain an estimate for $\left|\left\{|Tb^{(N)}|>\frac{\lambda}{4}\right\}\right|$ that is independent of $N$.  Indeed,
\begin{align*}
\left|\left\{|Tb|>\frac{\lambda}{2}\right\}\right| \leq \bigg|\bigg\{|T(b-b^{(N)})|>\frac{\lambda}{4}\bigg\}\bigg|+\bigg|\bigg\{|Tb^{(N)}|>\frac{\lambda}{4}\bigg\}\bigg|,
\end{align*}
while Chebyshev's inequality and the strong-type $(2,2)$ bound for $T$ imply
\begin{align*}
 \bigg|\bigg\{|T(b-b^{(N)})|>\frac{\lambda}{4}\bigg\}\bigg| \leq \frac{16}{\lambda^2} C_2 \|\Omega\|_{L^1(S^{n-1},\sigma)} \int_{\mathbb{R}^n} |b(x)-b^{(N)}(x)|^2\,dx.
\end{align*}
By the assumptions on $f$, both $b^{(N)}$ and $b$ are bounded with compact support, and so the pointwise convergence $b^{(N)}\to b$ and Lebesgue's dominated convergence theorem imply that this term tends to zero as $N \to \infty$.  This completes the proof of the claim.

Let $c_k$ denote the center of $Q_k$, let $a_k := \int_{\mathbb{R}^n}b_k(x)\,dx$, and let $\nu_N := \sum_{k=1}^Na_k\delta_{c_k}$. Then
\begin{align*}
	\bigg|\bigg\{|Tb^{(N)}|>\frac{\lambda}{4}\bigg\}\bigg|&\leq \bigg|\bigg\{|T(b^{(N)}dm-\nu_N)|>\frac{\lambda}{8}\bigg\}\bigg|+\bigg|\bigg\{|T\nu_N|>\frac{\lambda}{8}\bigg\}\bigg| \\
	&\leq |U|+\bigg|\bigg\{x\in \mathbb{R}^n\setminus U:|T(b^{(N)}dm-\nu_N)(x)|>\frac{\lambda}{8}\bigg\}\bigg|+\bigg|\bigg\{|T\nu_N|>\frac{\lambda}{8}\bigg\}\bigg|,
\end{align*}
where $dm$ represents the Lebesgue measure. Using property (2), we have
\begin{align*}
	|U|=\sum_{k=1}^{\infty}|Q_k|\leq \frac{1}{\lambda}\|f\|_{L^1(\mathbb{R}^n)}.
\end{align*}

To estimate the second term, we apply Chebyshev's inequality, Lemma \ref{cancellation}, and property (3) to obtain
\begin{align*}
\bigg|\bigg\{x \in \mathbb{R}^n\setminus U: |T(b^{(N)}dm-\nu_N)(x)|>\frac{\lambda}{8}\bigg\}\bigg| &\leq \frac{8}{\lambda}\int_{\mathbb{R}^n\setminus U}|T(b^{(N)}dm-\nu_N)(x)|\,dx\\
&\leq\frac{8}{\lambda}\sum_{k=1}^N\int_{\mathbb{R}^n\setminus U} |T(b_kdm-a_k\delta_{c_k})(x)|\,dx\\
&\leq 8C_1\|\Omega \|_{L^1(S^{n-1},\sigma)}\frac{1}{\lambda}\sum_{k=1}^N\|b_kdm-a_k\delta_{c_k}\|_{M_b(\mathbb{R}^n)}\\
&\leq 16C_1\|\Omega \|_{L^1(S^{n-1},\sigma)}\frac{1}{\lambda}\|b\|_{L^1(\mathbb{R}^n)}\\
&\leq 16C_1\|\Omega \|_{L^1(S^{n-1},\sigma)}\frac{1}{\lambda}\|f\|_{L^1(\mathbb{R}^n)}.
\end{align*}

Collecting the previous estimates, we conclude
\begin{align*}
	|\{|Tf|>\lambda\}| &\leq \left(4C_2^2+1+16C_1\|\Omega \|_{L^1(S^{n-1},\sigma)}\right)\frac{1}{\lambda}\|f\|_{L^1(\mathbb{R}^n)} + \bigg|\bigg\{|T\nu_N|>\frac{\lambda}{8}\bigg\}\bigg| \\
	&\leq \left(C_4'+C_5'\|\Omega\|_{L^1(S^{n-1},\sigma)}\right)\frac{1}{\lambda}\|f\|_{L^1(\mathbb{R}^n)}+\sup_{\nu}|\{|T\nu|>\lambda/8\}|,
\end{align*}
where $C_4' = 4C_2^2+1$ and $C_5' = 16C_1$. The argument is thus complete in the case of nonnegative $f$. 

In the case where $f$ is signed, and to obtain the constants we claim in the statement of the theorem, we write $f=f^+-f^-$, and estimate
\begin{align*}
|\{|Tf|>\lambda\}| \leq |\{|Tf^+|>\lambda/2\}|+|\{|Tf^-|>\lambda/2\}|.
\end{align*}
The preceding argument can then be applied to the two terms separately, and allows us to conclude the theorem with $C_4=2C_4'$, $C_5=2C_5'$, and noting that this is where we obtain the constant $2$ in the term
\begin{align*}
2 \sup_{\nu}|\{|T\nu|>\lambda\}|
\end{align*}
and why the supremum is over $\nu$ such that $\sum_{k=1}^N a_k \leq 16\|f\|_{L^1(\mathbb{R}^n)}$.
\end{proof}

We now prove Theorem \ref{pointmassreduction}.

\begin{proof}[Proof of Theorem \ref{pointmassreduction}]
From Theorem \ref{pointmassreduction1} we have that
\begin{align*}
|\{|R_jf|>\lambda\}|\leq \left(C_4+C_5\|\Omega\|_{L^1(S^{n-1},\sigma)}\right)\frac{1}{\lambda}\|f\|_{L^1(\mathbb{R}^n)}+2\sup_{\nu}|\{|R_j\nu|>\lambda\}|,
\end{align*}
while the computation of Janakiraman \cite{J2004} referenced above in Example \ref{Riesz1} shows
\begin{align*}
\int_{S^{n-1}}|\Omega(\theta)|\, d\sigma(\theta) = \frac{2}{\pi}.
\end{align*}
Therefore the theorem holds with
\begin{align*}
C= C_4+C_5\frac{2}{\pi} + 2.
\end{align*}

\end{proof}

We next prove an auxiliary result that will be of use in our proof of Theorem \ref{pointmassweaktype}.

\begin{thm}\label{pointmassweaktype1}
If $\Omega$ is an odd function, then there exist absolute constants $C_6,C_7,C_8>0$ such that 
\begin{align*}
	\sup_{\lambda>0}\lambda|\{|T\nu|>\lambda\}| \leq \left(C_6+\left(C_7  \max\left\{1,\|\Omega \|_{L^1(S^{n-1},\sigma)}\right\}+ C_8\log n\right)\|\Omega\|_{L^1(S^{n-1},\sigma)}\right) \|\nu\|_{M_b(\mathbb{R}^n)}
\end{align*}
for all measures $\nu$ of the form $\nu = \sum_{k=1}^N a_k\delta_{c_k}$ with $a_k \in \mathbb{R}^+$.
\end{thm}

\begin{proof}
Set
$$
	E_1:=B(c_1,r_1),
$$
where $r_1>0$ is chosen so that $|E_1|=\frac{a_1}{\lambda}$. In general, for $k \in\{2,\ldots,N\}$, set
$$
	 E_k:=B(c_k,r_k)\setminus \bigcup_{i=1}^{k-1}E_i,
$$ 
where $r_k>0$ is chosen so that $|E_k|=\frac{a_k}{\lambda}$. Define 
$$
	h := \sum_{k=1}^N \chi_{\mathbb{R}^n \setminus B(c_k,nr_k)}T\chi_{E_k}.
$$ 
and set 
$$
	E := \bigcup_{k=1}^NE_k.
$$ 

We have 
\begin{align*}
	|\{|T\nu|>\lambda\}|&\leq \bigg|\bigg\{|T\nu-\lambda h|>\frac{\lambda}{2}\bigg\}\bigg|+\bigg|\bigg\{|h|>\frac{1}{2}\bigg\}\bigg| \\
	&\leq |E|+ \bigg|\bigg\{x \in \mathbb{R}^n \setminus E : |T\nu(x)-\lambda h(x)|>\frac{\lambda}{2}\bigg\}\bigg| + \bigg|\bigg\{|h|>\frac{1}{2}\bigg\}\bigg| \\
&=:\text{I}+\text{II}+\text{III}.
\end{align*}

Since the $E_k$ are pairwise disjoint we have
$$
	\text{I}=\sum_{k=1}^N |E_k|=\frac{1}{\lambda}\sum_{k=1}^Na_k=\frac{1}{\lambda}\|\nu\|_{M_b(\mathbb{R}^n)}.
$$

To control $\text{II}$, first notice 
$$
	T\nu-\lambda h = \sum_{k=1}^N \chi_{\mathbb{R}^n \setminus{B(c_k,nr_k)}}T(a_k\delta_{c_k}-\lambda\chi_{E_k}dm)+\sum_{k=1}^Na_k\chi_{{B(c_k,nr_k)}}T\delta_{c_k}.
$$
By Chebyshev's inequality one has
\begin{align*}
	\text{II} &\leq \frac{2}{\lambda}\int_{\mathbb{R}^n\setminus E}|T\nu(x)-\lambda h(x)|\,dx \\
	&\leq\frac{2}{\lambda}\sum_{k=1}^N\int_{\mathbb{R}^n\setminus {B(c_k,nr_k)}}|T(a_k\delta_{c_k}-\lambda \chi_{E_k}dm)(x)|\,dx + \frac{2}{\lambda}\sum_{k=1}^Na_k\int_{{B(c_j,nr_k)}\setminus B(c_k,r_k)}|T\delta_{c_k}(x)|\,dx.
\end{align*}
Meanwhile, Lemma \ref{cancellation} implies
\begin{align*}
	\frac{2}{\lambda}\sum_{k=1}^N\int_{\mathbb{R}^n\setminus {B(c_k,nr_k)}}|T(a_k\delta_{c_k}-\lambda\chi_{E_k}dm)(x)|\,dx \leq \frac{2C_1}{\lambda}\sum_{k=1}^{N}\|a_k\delta_{c_k}-\lambda\chi_{E_k}dm\| \leq \frac{4C_1}{\lambda}\|\nu\|_{M_b(\mathbb{R}^n)}.
\end{align*}
Noticing that $T\delta_{c_k}(x)=K(x-c_k)$ and integrating with polar coordinates, we have
\begin{align*}
	\frac{2}{\lambda}\sum_{k=1}^{N}a_k\int_{{B(c_k,nr_k)}\setminus B(c_k,r_k)}|T\delta_{c_k}(x)|\,dx &\leq \frac{2}{\lambda}\sum_{k=1}^N a_k\int_{{B(c_k,nr_k)}\setminus B(c_k,r_k)}\frac{|\Omega(x-c_k)|}{|x-c_k|^n}\,dx\\
&=\frac{2}{\lambda}\sum_{k=1}^{N}a_k\int_{S^{n-1}}|\Omega(\theta)|\int_{r_k}^{nr_k}\frac{1}{t^n}t^{n-1}\,dtd\sigma(\theta)\\
&=2\log(n)\bigg(\int_{S^{n-1}}|\Omega(\theta)|\,d\sigma(\theta)\bigg)\frac{1}{\lambda}\|\nu\|_{M_b(\mathbb{R}^n)}.
\end{align*}
Therefore
\begin{align*}
	\text{II}\leq \bigg(4C_1+2\log(n)\bigg(\int_{S^{n-1}}|\Omega(\theta)|\,d\sigma(\theta)\bigg)\bigg)\frac{1}{\lambda}\|\nu\|_{M_b(\mathbb{R}^n)}.
\end{align*}

We next bound $\text{III}$. Since $\text{III}\leq |\{h>\frac{1}{4}\}|+|\{h<-\frac{1}{4}\}|$ we will just bound the first term and note that the same bound holds for the second term. Take a compact set $F \subseteq \{h>\frac{1}{4}\}$. Then Chebyshev's inequality yields the bound 
$$
	\frac{1}{4}|F|< \int_F h(x)\,dx.
$$ 

We will bound $\int_F h(x)\,dx$ above. First, we move to the adjoint of $T$, $T^*$: 
$$
	\int_F h(x) \, dx=\int_{\mathbb{R}^n}\bigg(\sum_{k=1}^{N}\chi_{F\setminus B(c_k,nr_k)}T\chi_{E_k}(x)\bigg)\,dx=\sum_{k=1}^N\int_{E_k}T^*\chi_{F\setminus B(c_k,nr_k)}(x)\,dx.
$$
Next, we add and subtract $T^*\chi_{F\setminus B(x,(n-1)r_k)}(x)$ and apply the triangle inequality to obtain
\begin{align*}
	|T^*\chi_{F\setminus B(c_k,nr_k)}(x)| \leq |T^*\chi_{F\setminus B(c_k,nr_k)}(x)-T^*\chi_{F\setminus B(x,(n-1)r_k)}(x)|+|T^*\chi_{F\setminus B(x,(n-1)r_k)}(x)|.
\end{align*}
For the first term, we use the fact $F\cap(B(c_k,nr_k)\setminus B(x,(n-1)r_k)) \subseteq B(x,(n+1)r_k)\setminus B(x,(n-1)r_k)$ and integrate in polar coordinates:
\begin{align*}
|T^*\chi_{F\setminus B(c_k,nr_k)}(x)-T^*\chi_{F\setminus B(x,(n-1)r_k)}(x)|&\leq \int_{F\cap(B(c_k,nr_k)\setminus B(x,(n-1)r_k))}|K(y-x)|\,dy\\
&\leq \int_{B(x,(n+1)r_k)\setminus B(x,(n-1)r_k)}\frac{|\Omega(y-x)|}{|y-x|^n}\,dy\\
&=\bigg(\int_{S^{n-1}}|\Omega(\theta)|\,d\sigma(\theta)\bigg)\int_{(n-1)r_k}^{(n+1)r_k}\frac{1}{t^n}t^{n-1}\,dt\\
&\leq\log(3)\|\Omega \|_{L^1(S^{n-1},\sigma)}.
\end{align*}
Therefore 
\begin{align*}
	\sum_{k=1}^ N\int_{E_k}|T^*\chi_{F\setminus B(c_k,nr_k)}(x)-T^*\chi_{F\setminus B(x,(n-1)r_k)}(x)|\,dx \leq \log(3)\|\Omega \|_{L^1(S^{n-1},\sigma)} |E|.
\end{align*}
For the second term, by Lemma \ref{dimensionlessl2} and the Cauchy-Schwarz inequality we find
\begin{align*}
	\sum_{k=1}^N\int_{E_k}|T^*\chi_{F\setminus B(x,(n-1)r_k)}&(x)|\,dx\leq \int_E(T^*)^{\#}\chi_F(x)\,dx\\
&\leq |E|^{\frac{1}{2}}\|(T^*)^{\#}\chi_F\|_{L^2(\mathbb{R}^n)}\\
&\leq C_3\|\Omega \|_{L^1(S^{n-1},\sigma)}|E|^{\frac{1}{2}}|F|^{\frac{1}{2}}.
\end{align*}

Therefore
$$
	\frac{1}{4}|F|\leq \|\Omega \|_{L^1(S^{n-1},\sigma)} \left( \log(3) |E|+  C_3|E|^{\frac{1}{2}}|F|^{\frac{1}{2}}\right).
$$
If $|F|\leq |E|$, we obtain
$$
	|F|\leq 4\|\Omega \|_{L^1(S^{n-1},\sigma)}\left(\log(3) +  C_3\right)|E|.
$$
Otherwise, $|E|\leq |E|^{\frac{1}{2}}|F|^{\frac{1}{2}}$ and therefore 
$$
	\frac{1}{4}|F|^{\frac{1}{2}}\leq  \|\Omega \|_{L^1(S^{n-1},\sigma)}\left( \log(3) |E|^{\frac{1}{2}}+  C_3|E|^{\frac{1}{2}}\right),
$$
which says 
$$
	|F|\leq  32\|\Omega \|_{L^1(S^{n-1},\sigma)}^2 \left(\log(3)^2 + C_3^2\right) |E|.
$$
In either case, taking the supremum over such $F$ implies 
$$
	\bigg|\bigg\{h>\frac{1}{4}\bigg\}\bigg| \leq B \max\left\{\|\Omega \|_{L^1(S^{n-1},\sigma)},\|\Omega \|_{L^1(S^{n-1},\sigma))}^2\right\} \frac{1}{\lambda}\|\nu\|_{M_b(\mathbb{R}^n)},
$$
where 
$$
	B:=32 \bigg(\log(3)^2 +\max\{C_3^2,C_3\}\bigg).
$$
Hence 
$$
	\text{III}\leq 2B \max\left\{\|\Omega \|_{L^1(S^{n-1},\sigma)},\|\Omega \|_{L^1(S^{n-1},\sigma))}^2\right\} \frac{1}{\lambda}\|\nu\|_{M_b(\mathbb{R}^n)}.
$$
Putting the above estimates together, we obtain
\begin{align*}
	\sup_{\lambda>0}\lambda&|\{|T\nu|>\lambda\}|\leq \\&\left(1+4C_1+ \left(2B \max\left\{1,\|\Omega \|_{L^1(S^{n-1},\sigma)}\right\}+2\log(n)\right)\|\Omega \|_{L^1(S^{n-1},\sigma)} \right)\|\nu\|_{M_b(\mathbb{R}^n)}.
\end{align*}
We conclude the theorem with 
\begin{align*}
C_6&=1+4C_1,\\
C_7 &=2B,\quad\text{and} \\
C_8 &= 2.
\end{align*}
\end{proof}

We finish with a proof of Theorem \ref{pointmassweaktype}.  

\begin{proof}[Proof of Theorem \ref{pointmassweaktype}]
By Theorem \ref{pointmassweaktype1}, we have 
\begin{align*}
\sup_{\lambda>0}\lambda|\{|R_j\nu|>\lambda\}| \leq \left(C_6+\left(C_7  \max\left\{1,\|\Omega \|_{L^1(S^{n-1},\sigma)}\right\}+ C_8\log n\right)\|\Omega\|_{L^1(S^{n-1},\sigma)}\right) \|\nu\|_{M_b(\mathbb{R}^n)}
\end{align*}
while the observation that
\begin{align*}
\|\Omega \|_{L^1(S^{n-1},\sigma)} = \frac{2}{\pi}<1
\end{align*}
yields
\begin{align*}
	\sup_{\lambda>0}\lambda|\{|R_j\nu|>\lambda  \}| &\leq \left(C_6+C_7   + C_8\log n\right) \frac{1}{\lambda} \|\nu\|_{M_b(\mathbb{R}^n)}.
\end{align*}
Therefore the theorem holds with
\begin{align*}
C=\frac{C_6+C_7}{\log(2)}+ C_8.
\end{align*}
\end{proof}


\section*{Acknowledgments} 
Part of this work was written while the second author (CBS) was visiting the Nonlinear Analysis Unit in the Okinawa Institute of Science and Technology Graduate University.  He warmly thanks OIST for the invitation and hospitality.

The authors would also like to thank Brett Wick for inspiring conversations regarding this work.


\end{document}